\documentclass[a4paper,12pt]{amsart}
\usepackage[utf8]{inputenc}

\usepackage{amsmath,amssymb,amsthm}
\usepackage{bbm}
\usepackage{bm}
\usepackage{xspace}
\usepackage{csquotes}
\usepackage{standalone}

\usepackage[x11names,table]{xcolor}
\usepackage[naturalnames=false]{hyperref}
\usepackage{graphicx}
\graphicspath{ {img/} }
\usepackage{tikz}
\usetikzlibrary{calc}
\usetikzlibrary{cd}
\usetikzlibrary{shapes}
\usepackage{csquotes}
\usepackage{booktabs}
\hypersetup{
	colorlinks = true,          
	urlcolor = DeepSkyBlue3, linkcolor = Chartreuse4, citecolor = DarkOrange2
}

\newcommand \acknowledgements{\subsection*{Acknowledgements}}

\RequirePackage[backend=bibtex,isbn=false,url=false,bibencoding=utf8,maxbibnames=4]{biblatex}
\renewbibmacro{in:}{, }
\DeclareFieldFormat[misc]{title}{\mkbibquote{#1}}
\DeclareFieldFormat[article]{volume}{\mkbibbold{#1}}
\DeclareFieldFormat{url}{\textsc{url}: \href{#1}{#1}}
\DeclareFieldFormat{doi}{\textsc{doi}: \href{http://dx.doi.org/#1}{#1}}
\DeclareFieldFormat{eprint:arxiv}{arXiv:\href{https://arxiv.org/abs/#1}{#1}}

\newif\iffpsac
\fpsacfalse

\newcommand\addressmark[1]{\relax}

\address[Andrei Com\u{a}neci]{
 	Technische Universit\"{a}t Berlin,
 	Chair of Discrete Mathematics/Geometry \\
 	\texttt{comaneci@math.tu-berlin.de}	
}

\address[Michael Joswig]{
 	Technische Universit\"{a}t Berlin,
 	Chair of Discrete Mathematics/Geometry; \\
 	Max-Planck Institute for Mathematics in the Sciences, Leipzig \\
 	\texttt{joswig@math.tu-berlin.de}
}

\usepackage{graphicx}
\graphicspath{ {img/} }
\usepackage{tikz}
\usetikzlibrary{calc}
\usetikzlibrary{cd}
\usetikzlibrary{shapes}
\usepackage{csquotes}
\usepackage{booktabs}
\usepackage{algorithm}
\usepackage{algpseudocode}

\usepackage{standalone}

\usepackage{mathbbol}

\newcommand\NN{{\mathbb N}}

\newcommand\PP{{\mathbb P}}

\newcommand\RR{{\mathbb R}}

\newcommand\TT{{\mathbb T}}
\newcommand\ZZ{{\mathbb Z}}

\newcommand\cS{{\mathcal S}}

\newcommand\cL{{\mathcal L}}

\newcommand\cT{{\mathcal T}}

\newcommand\frT{{\mathfrak T}}

\DeclareMathSymbol{\0}{\mathord}{bbold}{`0}
\DeclareMathSymbol{\1}{\mathord}{bbold}{`1}

\newcommand\smallSetOf[2]{\{#1 \mid #2\}}

\newcommand\torus[1]{\RR^{#1}/\RR\1} 

\newcommand\TP{{\TT\PP}}

\DeclareMathOperator\conv{conv}
\DeclareMathOperator\tconv{tconv}

\DeclareMathOperator\nvol{nvol}

\DeclareMathOperator\argmin{arg\,min}
\DeclareMathOperator\argmax{arg\,max}

\DeclareMathOperator\relint{relint}
\DeclareMathOperator\lca{lca}
\DeclareMathOperator\Tr{Tr}

\DeclareMathOperator{\FW}{FW} 
\DeclareMathOperator\CovDec{CovDec} 
\DeclareMathOperator\TGr{TGr} 

\newcommand\dd[1]{\operatorname{d}_{#1}} 

\newcommand\supp{\text{supp}}

\newcommand\transpose[1]{#1^\top}

\newtheorem{theorem}{Theorem}
\newtheorem{lemma}[theorem]{Lemma}
\newtheorem{proposition}[theorem]{Proposition}
\newtheorem{corollary}[theorem]{Corollary}

\theoremstyle{remark}

\newtheorem{example}[theorem]{Example}

\newtheorem{question}[theorem]{Question}

\definecolor{tolorange}{RGB}{238,119,51}  
\definecolor{tolblue}{RGB}{0,119,187}     
\definecolor{tolgreen}{RGB}{0,153,136}    
\definecolor{tolpurple}{RGB}{136,34,85}   
\definecolor{ibmyellow}{RGB}{255, 176, 0} 
\definecolor{darkgray}{RGB}{64,64,64} 

\tikzstyle{blackdot} = [circle,draw=black,fill=black,scale=0.5]

\usepackage{expl3}
\ExplSyntaxOn
\newcommand \breakDOI[1]
 {
  \tl_map_inline:nn { #1 } { \href{http://dx.doi.org/#1}{##1} \penalty0 \scan_stop: }
 }
\ExplSyntaxOff

\DeclareFieldFormat{doi}{%
  \mkbibacro{DOI}\addcolon\space
  {\breakDOI{#1}}}

\title{Parametric Fermat--Weber and tropical supertrees}

\author{Andrei Com\u{a}neci\addressmark{1} \and Michael Joswig\addressmark{1,2}}

\newcommand\thisabstract{%
  We study a parametric version of the Fermat--Weber problem with respect to an asymmetric distance function, which occurs naturally in tropical geometry.
  Our results yield a method for constructing phylogenetic supertrees.}
\iffpsac
\abstract{\thisabstract}
\fi

\begin{document}

\maketitle

\iffpsac
\else
\begin{abstract}
  \thisabstract
\end{abstract}
\fi

\section{Introduction}
\noindent
The Fermat--Weber problem is a classical location problem that seeks a point whose average distance to a finite set of sites is minimal; such a point is called a \emph{Fermat--Weber point}.
The history of the problem can be traced back to Fermat and gained popularity in applied context under the work of Weber; see \cite[Chapter~II]{BMS:1999} for a historical account.
In \cite{ComaneciJoswig:2205.00036} we looked at the Fermat--Weber problem for an asymmetric tropical distance function.
In this paper, we focus on a parametric version of the tropical Fermat--Weber problem; see \cite{Brandeau+Chiu:91} for a survey on general parametric location problems.
More precisely, we consider the case where each site moves along a line, and we study how the set of Fermat--Weber points changes.
The results in \cite{ComaneciJoswig:2205.00036} heavily exploit the known structure of tropically convex sets in terms of regular subdivisions of products of simplices \cite{Develin+Sturmfels:2004}; see also \cite[\S6.3]{ETC}.
Here we analyze the relevant secondary fans in more detail, by looking at the variation of regular subdivisions with respect to changing the height function.
In this way we show that the Fermat--Weber set asymptotically stabilizes from a combinatorial point of view.
Further, we give an explicit bound for that stabilization to occur.
This is our first contribution.

Our study of tropical Fermat--Weber points is motivated by phylogenetics \cite{SempleSteel:2003}.
That field of computational biology is concerned with organizing ancestral relations in biological data into trees.
In our scenario, the \emph{taxa}, which are the individual points in a given data set, are associated with the leaves of a tree; the root marks the most recent common ancestor.
A basic problem, which occurs in practice, asks to compile the information from several trees into a single tree.
In the easier case the individual trees all share the same taxa, and the resulting tree is called a \emph{consensus tree}.
Due to the tremendous interest in phylogenomics methods, and also to differing demands in different application scenarios, many methods were developed over the last few decades \cite{Kapli+Yang+Telford:2020}.
They all output only estimates of the true phylogeny and thus it is necessary to evaluate the confidence of the results.
Apart from Bayesian methods, where posterior distributions yield measures for confidence, the usual approach is to use bootstrapping generating solution sets.
In any case, assessment usual involves consideration of many phylogenetic trees and the search for a consensus.
Similarly, trees resulting from distinct methods, for a fixed number of taxa, often disagree, and again finding a consensus is of central interest.
Since this turns out to be a nontrivial task, computing consensus trees became a topic of its own  \cite{Bryant:2003,propSupertree,Bryant+Francis+Steel:2017}.

In \cite{ComaneciJoswig:2205.00036} we constructed a consensus tree as a Fermat--Weber point in tree space.
In a more general setting, the input trees may have distinct sets of taxa; e.g., see \cite{li20:_phylog_sars_cov}.
Here the resulting tree is called a \emph{supertree}, and our second contribution is a new method for their construction.
These \emph{tropical supertrees} arise from combining the tropical consensus trees from \cite{ComaneciJoswig:2205.00036} with our new parametric solution of the tropical Fermat--Weber problem and the \emph{big-M method} in optimization.
The latter technique explores the behavior of the feasible domain of an optimization problem at infinity by computing with sufficiently large real values; see, e.g., \cite[\S11.2]{Schrijver:1986}.


\acknowledgements
We want to thank Heike Siebert for valuable discussions on phylogenomics.
Support by the Deutsche Forschungsgemeinschaft (DFG, German Research Foundation) gratefully acknowledged.
AC has been supported by \enquote{Facets of Complexity} (GRK 2434, project-ID 385256563).
MJ has been supported by \enquote{Symbolic Tools in Mathematics and their Application} (TRR 195, project-ID 286237555).

\section{Parametric tropical medians}\label{sec:parametric}
\noindent
Before we can look at algorithms for phylogenetics, we need to investigate specific questions concerning covector decompositions in tropical convexity; see \cite[\S6.3]{ETC}.
Later metric trees will occur as points in a suitable tropical projective torus, and a tropical supertree is a tree in a certain cell in the covector decomposition generated by those points.

A nonempty set $S\subseteq\RR^n$ is \emph{tropically convex} (with respect to $\max$) if $z$ lies in $S$, where $z_i=\max(\lambda+x_i,\mu+y_i)$, for all $x,y\in S$ and $\lambda,\mu\in\RR$.
Each tropically convex subset of $\RR^n$ contains $\RR\1$, and hence it makes sense to consider tropically convex sets in the quotient $\torus{n}$, which is the \emph{tropical projective torus}; here $\1$ denotes the all ones vector.
The \emph{tropical convex hull} of $S$ is the smallest tropically convex set containing $S$, and a \emph{tropical polytope} is the tropical convex hull of finitely many points in $\torus{n}$.
It may happen that a tropical polytope is a convex subset in the ordinary sense; then it is a \emph{polytrope}; cf.\  \cite[\S6.5]{ETC}.
Any finite point set $V\subset\torus{n}$ induces a polyhedral decomposition of $\torus{n}$, the \emph{covector decomposition} $\CovDec(V)$, such that the bounded cells comprise the tropical convex hull $\tconv(V)$.
Each cell of the covector decomposition is a polytrope.
The matrix $V$ in $\RR^{m\times n}$ may also be seen as a lifting function on the product of simplices $\Delta_{m-1}\times\Delta_{n-1}$.
This gives rise to a regular subdivision, $\Sigma(V)$; see \cite[\S1.2]{ETC} for a brief introduction and \cite{DLRS} for more comprehensive information.
It turns out that $\Sigma(V)$ is dual to $\CovDec(V)$ \cite[Corollary~6.15]{ETC}.

We consider the \emph{asymmetric tropical distance} function $\dd{n}$ on $\RR^n$, which we define as
\begin{equation}\label{eq:dist}
  \dd{n}(x,y)=\sum_i (x_i - y_i) - n\cdot\min_j(x_j - y_j) \enspace.
\end{equation}
This descends to the quotient $\torus{n}$.
The function $\dd{n}$ is the convex distance function induced by the standard simplex $\Delta_{n-1}=\conv(e_1,\dots,e_n)$, where the vectors $e_i$ form the standard basis of $\RR^n$.
Since the polytope $\Delta_{n-1}$ has codimension one in $\RR^n$, passing to $\torus{n}$ is natural.

The crucial algorithmic ingredient is the following optimization problem that we studied in \cite{ComaneciJoswig:2205.00036}.
Given a set, $V$, of $m$ points in $\torus{n}$ we want to find a point $x\in\torus{n}$ such that $\sum_{v\in V}\dd{n}(v,x)$ is minimal.
Such a point, in general, is not unique, and the set of all minimizers is the \emph{asymmetric tropical Fermat--Weber set}, denoted $\FW(V)$.
We can identify minimizers by looking at the unique cell of $\Sigma(V)$ that contains the vertex barycenter $(\frac{1}{m}\1,\frac{1}{n}\1)$ of $\Delta_{m-1}\times\Delta_{n-1}$.
We name it the \emph{central cell} of $\Sigma(V)$ and its dual in $\CovDec(V)$ the \emph{central covector cell}.
We recall the following result.
\begin{theorem}[{\cite[Theorem~3]{ComaneciJoswig:2205.00036}}]\label{thm:FW}
  The Fermat--Weber set $\FW(V)$ agrees with the central covector cell in $\CovDec(V)$.
  In particular, $\FW(V)$ is a bounded polytrope in $\torus{n}$, and it is contained in the tropical polytope $\tconv(V)$.
\end{theorem}

Our first goal is to generalize Theorem~\ref{thm:FW} to the situation where the point configuration $V$ depends on a parameter.
To this end we consider points $u_1,\dots,u_m,w_1,\dots,w_m\in\RR^n$, and we call the $m$ affine lines
\begin{equation}\label{eq:affine-family}
  u_1+M w_1 \,,\ u_2+M w_2 \,,\ \dots \,,\ u_m+M w_m
\end{equation}
in $\RR^n$ an \emph{affine family} of point configurations, depending on a real parameter $M$.
The affine lines \eqref{eq:affine-family} may be seen as an affine family in $\torus{n}$.
Picking a specific value $M_0\in\RR$ we set $v_i=v_i(M_0)=u_i+M_0w_i$, and $V(M_0)=\{v_1,\dots,v_m\}$ is a usual configuration of $m$ points in $\RR^n$ or $\torus{n}$.
Often we write $V(M)=U+M\cdot W\in\RR^{m\times n}$ as a matrix, and $v_i$ is the $i$th row.
Our refined goal is now to describe the Fermat--Weber set $\FW(V(M))$ for $M$ sufficiently large.
Following standard practice, and to avoid cumbersome notation, we intentionally blur the line between $M$ as a real number and $M$ as a formal parameter.
The name \enquote{$M$} is chosen as a hint to the big-M method in optimization; see, e.g., \cite[\S11.2]{Schrijver:1986}.
We may view an affine family as a set of points, where each point travels along an affine line (into the positive direction).

To understand the behavior of $\CovDec(V(M))$ with the change of $M$ we make use of the duality to the regular subdivision $\Sigma(V(M))$ of $\Delta_{m-1}\times\Delta_{n-1}$.
There are only finitely many regular subdivisions of $\Delta_{m-1}\times\Delta_{n-1}$, and the lifting function corresponding to any fixed subdivision, $\Sigma$, form a relatively open polyhedral cone, called the \emph{relative open secondary cone} of $\Sigma$.
It follows that the height functions on $\Delta_{m-1}\times\Delta_{n-1}$ give rise to a complete polyhedral fan in the vector space  $\RR^{m\times n}$; this is the \emph{secondary fan} of $\Delta_{m-1}\times\Delta_{n-1}$.
The finiteness of the secondary fan implies that there exists $M_0\in\RR$ such that for all real $M_1,M_2>M_0$ the point configurations $V(M_1)$ and $V(M_2)$ lie in the same secondary cone, and thus $\Sigma(V(M_1))=\Sigma(V(M_2))$.
We say that the regular subdivisions $\Sigma(V(M))$ \emph{stabilize} for $M\gg 0$, and we use the same term for other polyhedra or polyhedral complexes depending on $M$.

Let $C$ be a covector cell of $\CovDec(V(M))$ and $C^\vee$ its dual cell in $\Sigma(V(M))$.
We can identify the cell $C^\vee$ with a bipartite graph $G(C)$ on the node set $[m]\sqcup [n]$, such that $(i,j)$ is an edge for $i\in[m]$ to $j\in[n]$ if and only if the vertex $(e_i,e_j)$ belongs to $C$; cf.\ \cite[\S4.7]{ETC}.
The graph bipartite $G(C)$ is called the \emph{covector graph} of the cell $C$.
We emphasize that for every $x\in C$ and every edge $(i,j)$ in $G(C)$ we have $v_{ij}(M)-x_j=\max_{k\in[n]}(v_{ik}(M)-x_k)$; cf. \cite[\S6.3]{ETC}.

In order to understand how the Fermat--Weber set $\FW(V(M))$ behaves with changing $M$, it suffices to track the tropical vertices.
By \cite[Theorem~6.38]{ETC} any polytrope in $\torus{n}$ has at most $n$ tropical vertices.
Each tropical vertex, $t_i(M)$, can be obtained by maximizing the linear functional $n\cdot e_i-\1$ on the Fermat--Weber set $\FW(V(M))$, where $i\in[n]$; see \cite[Theorem 3.26]{ETC}.
Note that duplicates may occur when $\FW(V(M))$ is not full-dimensional.
For our application to phylogenetics below, we want to pick a point in the relative interior of $\FW(V(M))$ consistently.
In \cite{ComaneciJoswig:2205.00036} we chose the ordinary average of the tropical vertices, i.e., the average of the distinct values of the set $\smallSetOf{t_i(M)}{i\in[n]}$.
The latter point is the \emph{tropical median} of the configuration $V(M)$.
The following result is the desired parametric version of Theorem~\ref{thm:FW}.

\begin{proposition} \label{prop:affDep}
  The combinatorial types of covector decompositions $\CovDec(V(M))$ stabilize for $M\gg 0$.
  That is, there exists $M_0>0$ such that $\CovDec(V(M))$ does not depend on $M$, provided that $M>M_0$.
  Moreover, for all $M>M_0$, the tropical vertices of $\FW(V(M))$ and the tropical median depend affinely on $M$.
  That is, there exist $p_i,q_i\in\torus{n}$, for $i\in[n]$, such that $t_i(M)=p_i+Mq_i$ for all $M>M_0$, and a similar statement holds for the tropical median.
\end{proposition}

It is worth pointing out that the same $M_0$ works for the covector decomposition, the tropical vertices and the tropical median.

\begin{proof}
  Since the regular subdivisions $\Sigma(V(M))$ stabilize for $M\gg 0$, it follows that the combinatorial types of covector decompositions $\CovDec(V(M))$ stabilize for $M\gg 0$, as they are dual to $\Sigma(V(M))$.
  According to Theorem~\ref{thm:FW}, the Fermat--Weber set $\FW(V(M))$ corresponds to the central cell of $\CovDec(V(M))$, so the stabilization of the covector decomposition implies that all tropical vertices of $\FW(V(M))$ share the same covector graph for $M$ sufficiently large.

  Consider the tropical vertex $x=x(M)=t_i(M)$, for some $i\in[n]$, and let $G$ be its covector graph for $M \gg 0$.
  Since $x$ is a vertex of $\CovDec(V(M))$, by \cite[Corollary 6.26]{ETC}, the graph $G$ is connected.
  So there exists a walk $W$ in $G$ visiting every coordinate node of $G$.
  To complete our proof we will proceed inductively along the walk $W$.

  To this end we fix $x_p=0$ for the first coordinate node $p\in[n]$ which occurs in $W$, and our goal is to recover all the other coordinates of $x$ by moving along the walk $W$.
  Let $x_r$ be a coordinate yet unknown, but we already have the value $x_q$ of the coordinate node $q$ which occurs before $r$ in $W$.
  By the induction hypothesis, we assume that $x_q$ depends affinely of $M$.
  Then there exists $j\in[m]$ such that $(q,j)$ and $(j,r)$ are arcs in $W$.
  This yields $v_{jq}(M)-x_q(M)=v_{jr}(M)-x_r(M)$.
  Since $x_q(M)$, $v_{jq}(M)$, and $v_{jr}(M)$ depend affinely on $M$, the coordinate $x_r(M)$ must also depend affinely on $M$.
\end{proof}

\begin{example} \label{ex:paramTM}
  Consider the affine family in $\torus{3}$ given by the three rows of the matrix 
  \[
    V(M) \ = \
    \begin{pmatrix}
      M+2  & -M-3  & 1  \\
      -6   & 2     & 4   \\
      -M+4 & -M-10 & 2M+6
    \end{pmatrix} \enspace .
  \]
  Here we pick representatives in $x+\RR\1$ such that $x_1+x_2+x_3=0$, as in \cite{ComaneciJoswig:2205.00036}.
  Then $\FW(V(M))$ agrees with $\tconv^{\max}(V(M))$ for $M\gg 0$, and so the tropical median is the average of the input points:
  \[ c(M)=\left(0,\, \frac{-2M-11}{3},\, \frac{2M+11}{3}\right) \enspace. \]
\end{example}

For computations it is useful to have an explicit threshold $M_0$ such that $\CovDec(V(M))$ has the same combinatorial type for any $M>M_0$.
To this end it will turn out to be helpful to gather specific information on the complexity of the secondary fan of the products of simplices $\Delta_{m-1}\times\Delta_{n-1}$.
For a (not necessarily regular) triangulation $\Sigma$ of $\Delta_{m-1}\times\Delta_{n-1}$ we define
\[
  \gamma_{ij} \ = \ \sum_{\begin{subarray}{c} \sigma \text{ maximal cell of } \Sigma \\ \text{with } e_i{\times}e_j \in\sigma \end{subarray}} \nvol(\sigma) \enspace,
\]
where $\nvol(\sigma)$ is the $(m{+}n{-2})$-dimensional normalized volume, which is integral.
The nonnegative integral matrix $\gamma(\Sigma)=(\gamma_{ij})_{i,j}\in\NN^{m\times n}$ is the \emph{GKZ-vector} of $\Sigma$; see \cite[Definition 5.1.6]{DLRS}.
The convex hull of the GKZ-vectors of all triangulations is known as the \emph{secondary polytope} of $\Delta_{m-1}\times\Delta_{n-1}$, and its vertices are precisely the GKZ-vectors of the regular triangulations; see \cite[Theorem 5.1.9]{DLRS}.

\begin{lemma} \label{lem:edgeSecPoly}
  Let $\Sigma$ be any (regular) triangulation of $\Delta_{m-1}\times\Delta_{n-1}$ with GKZ-vector $\gamma(\Sigma)=(\gamma_{ij})_{i,j}$.
  Then $\gamma_{ij}\leq\tbinom{m+n-2}{m-1}$ for all $i\in[m]$ and $j\in[n]$.
\end{lemma}

\begin{proof}
  Each triangulation of $\Delta_{m-1}\times\Delta_{n-1}$ is unimodular, and hence $\gamma_{ij}$ counts the number of maximal cells containing the vertex $e_i{\times}e_j$; see \cite[\S6.2.2]{DLRS}.
  It follows that the number of maximal cells of $\Sigma$ equals $\tbinom{m+n-2}{m-1}$, which is the normalized volume of $\Delta_{m-1}\times\Delta_{n-1}$.
\end{proof}

We return to studying our affine family $V(M)=U+M\cdot W$ with $U\in\RR^{m\times n}$, but we now restrict $W\in\ZZ^{m\times n}$ to be integral.
We denote the Euclidean scalar product of matrices as $\langle A,B\rangle=\sum_{i,j}a_{ij}b_{ij}=\Tr(A^\top B)$, and $\lVert A\rVert_\infty=\max_{i,j} \lvert a_{ij}\lvert$ is the $\infty$-norm, while $\lVert B\rVert_1=\sum_{i,j} \lvert b_{ij}\lvert$ is the $1$-norm.
Our first main result is the following quantitative version of Proposition~\ref{prop:affDep}.

\begin{theorem}\label{thm:affDep}
  For every $M>\tbinom{m+n-2}{m-1}\lVert U\rVert_1$, the covector decompositions $\CovDec(V(M))$ stabilize, and the tropical vertices of $\FW(V(M))$ as well as the tropical medians form affine families.
\end{theorem}

\begin{proof}
  We follow the argument of Proposition~\ref{prop:affDep} and pay attention to complexity issues.
  
  The argument employs the linear hyperplane arrangement in $\RR^{m\times n}$ formed by the facets of the maximal secondary cones of $\Delta_{m-1}\times\Delta_{n-1}$.
  Those cones bijectively correspond to the regular triangulations.
  Let $\Phi$ be such a facet, which is the common intersection of exactly two maximal secondary cones.
  The secondary fan is the normal fan of the secondary polytope, and hence a normal vector of $\Phi$ is the same as a scaled edge of the secondary polytope.
  The direction of such an edge is given by the difference of two GKZ-vectors.
  Thus it follows from Lemma~\ref{lem:edgeSecPoly} that there is a primitive normal vector, $A=(a_{ij})_{i,j}\in\ZZ^{m\times n}$, of $\Phi$ such that $|a_{ij}|\leq \tbinom{m+n-2}{m-1}$ for all $i,j$.
  We set $M_0:=\tbinom{m+n-2}{m-1}\lVert U\rVert_1$, and we have to show that that the sign of $\langle A,V(M)\rangle$ is constant for all $M>M_0$.
  If $\langle A,W\rangle=0$ then $\langle A,V(M)\rangle = \langle A, U+M\cdot W\rangle = \langle A,U\rangle$, and the claim is trivially valid.
  So we may assume that $\langle A,W\rangle\neq0$.

  Aiming for a contradiction, we further assume that there are distinct $M_1,M_3>M_0$ such that $\langle A,V(M_1)\rangle$ and $\langle A,V(M_3)\rangle$ have different signs.
  By continuity it follows that there is $M_2$ with $M_1\leq M_2 \leq M_3$ such that $\langle A,V(M_2)\rangle=0$ and $M_2>M_0$.
  Now the assumption $\langle A,W\rangle\neq0$ entails $M_2=-\langle A,U\rangle/\langle A,W\rangle$.
  We have
  \[
    |\langle A,U\rangle|=\left|\sum_{i,j}a_{ij}u_{ij}\right|\leq\lVert A \rVert_\infty\cdot\lVert U\rVert_1\leq \binom{m+n-2}{m-1}\lVert U\rVert_1 \enspace.
  \]
  This yields
  \[
    M_2 = -\frac{\langle A,U\rangle}{\langle A,W\rangle} \leq \frac{\lvert \langle A,U\rangle\rvert}{\lvert \langle A,W\rangle\rvert}\leq\binom{m+n-2}{m-1}\lVert U\rVert_1 = M_0 \enspace ,
  \]
  where the last inequality holds because $|\langle A,W\rangle|$ is a positive integer.
  That contradicts $M_2>M_0$, and this completes the proof.
\end{proof}


\section{Constructing supertrees}\label{sec:supertree}
\noindent
Now we are ready to explain our method for constructing supertrees.
They will arise as parameterized tropical medians of suitably modified equidistant trees.

We start out by describing our setup.
A rooted metric tree is \emph{equidistant} if the distance from the root to any leaf is the same.
The distance from the root to any leaf is the \emph{height} of the tree.
The latter quantity equals half of the maximal distance between any two leaves.
The leaves are labeled, and these labels are called \emph{taxa}.
Now we consider a collection $\cT=\{T_1,\dots,T_m\}$ of equidistant rooted metric trees, with labeled leaves, and we write $\cL(T_k)$ for the taxa of the tree $T_k$.
The combined set of taxa is $\cL=\bigcup_{k=1}^m\cL(T_k)$, and we let $n=\lvert\cL\rvert$ be the cardinality of $\cL$.
For consistency, we assume that the $m$ trees in $\cT$ share the same height.

A \emph{dissimilarity map} is a symmetric matrix $D=(D_{ij})_{i,j}$ with zero diagonal.
This is an \emph{ultrametric} if $D$ is nonnegative, and additionally the \emph{ultrametric inequality} $D_{ik} \leq \max (D_{ij} , D_{jk})$ is satisfied for all taxa $i,j,k$.
It is known that the ultrametrics are precisely the dissimilarity functions among the leaves of equidistant trees \cite[Theorem 7.2.5]{SempleSteel:2003}.

For each tree $T_\ell$ we define an ultrametric $D^{(\ell)}(M)=(D_{ij}^{(\ell)})_{i,j}$ on the set $\cL$ by setting
\begin{equation}\label{eq:distances}
  D_{ij}^{(\ell)} =
  \begin{cases}
    \text{distance in $T_\ell$} & \text{if $i,j\in\cL(T_\ell)$}\\
    M & \text{otherwise.}
  \end{cases}
\end{equation}
Here we assume that $M\gg0$ is a sufficiently large real number, so that $D^{(\ell)}(M)$ is, indeed, an ultrametric.

Due to Theorem~\ref{thm:FW}, the Fermat--Weber set $\FW(\cT)$, where each tree $T_\ell$ is represented by its ultrametric $D^{(\ell)}(M)$, is a polytrope in the tropical projective torus $\torus{\tbinom{\cL}{2}}$.
Ardila and Klivans \cite{ArdilaKlivans:2006} showed that the space of equidistant metric trees, $\frT(\cL)$, is max-tropically convex.
So, in view of Theorem~\ref{thm:FW}, $\FW(\cT)$ is actually contained in the subset $\frT(\cL)$.
Further, by Proposition~\ref{prop:affDep}, the tropical vertices of $\FW(\cT)$ depend affinely on $M$ for $M\gg0$, whence they trace out a ray in $\frT(\cL)$.
The Fermat--Weber set $\FW(\cT)$ is convex in the ordinary sense when we view $\torus{\tbinom{\cL}{2}}$ as a real vector space.
Therefore, $\FW(\cT)$ contains the tropical median, which is defined as the ordinary average of the (at most $\tbinom{n}{2}$) tropical vertices.
We call the tropical median, $c(\cT)$, in $\frT(\cL)$ the \emph{tropical supertree} of $\cT$.
The dissimilarity map associated with $c(\cT)$ depends on $M\gg0$, but we will show that the tree topology does not.
This will become our second main result.

Before we proceed, we need additional terminology from phylogenetics.
Three pairwise distinct taxa $i,j,k\in\cL$ form a \emph{rooted triplet}, denoted $ij|k$, of a rooted phylogenetic tree $T$ on $\cL$ if the lowest common ancestor, $\lca_T(i,j)$ of $i$ and $j$ is not an ancestor of $k$.
If $T$ is corresponds to the ultrametric $D$, then $ij|k$ is a rooted triplet if and only if $D_{ij}<D_{ik}$.
We write $r(T)$ (or $r(D)$) for the set of all rooted triplets.
Page, Yoshida, and Zhang \cite[Theorem~3.2]{Page+Yoshida+Zhang:2020} showed that any two trees in the relative interior of a covector cell share the same topology.
We use the techniques from \cite[\textsection 6.4]{SempleSteel:2003} to prove a vast generalization of their result.
\begin{proposition} \label{prop:ordConvTreeSpace}
  Let $\cS$ be an arbitrary convex subset of $\frT(\cL)$, in the ordinary sense.
  Then any two points in the relative interior of $\cS$ are equidistant trees with the same topology.
\end{proposition}

\begin{proof}
  By \cite[Theorem~6.4.1]{SempleSteel:2003} the tree topology is determined by the set of rooted triplets.
  That is, it suffices to prove that $r(D)=r(E)$ for any two ultrametrics $D,E\in\relint(\cS)$.
  To this end let $ij|k\in r(D)$, i.e., $D_{ij}<D_{ik}=D_{jk}$.

  Our first goal is to prove that $ik|j$ and $jk|i$ are not contained in $r(E)$.
  The condition $ij|k\in r(D)$ is symmetric in $i$ and $j$, whence it suffices to show $(ik|j)\not\in r(E)$.
  We assume the contrary, which is equivalent to $E_{jk}=E_{ij}>E_{ik}$.
  Due to the convexity of $\cS$, the point $F=\tfrac{1}{2}(D+E)$ belongs to $\cS\subset\frT(\cL)$.
  Yet this is not an ultrametric as
  \[
    F_{jk}=\frac{1}{2}(D_{jk}+E_{jk})>\frac{1}{2}(D_{ij}+E_{ij})=F_{ij}
  \]
  and, similarly, $F_{jk}>F_{ik}$.
  We conclude that either $(ij|k)\in r(E)$ or $E_{ij}=E_{ik}=E_{jk}$, and we need to exclude the latter case.
  Again we argue by contradiction.

  So we assume that $E_{ij}=E_{ik}=E_{jk}$.
  Then we consider $F=E+\epsilon(E-D)$, which lies in $\relint(\cS)$ for some sufficiently small $\epsilon>0$.
  This yields $F_{jk}=F_{ik}<F_{ij}$, which violates the ultrametric inequality.
  Hence, we infer that $(ij|k)\in r(E)$ and thus $r(D)\subseteq r(E)$.
  The reverse inclusion follows from exchanging $D$ with $E$.
\end{proof}

The Proposition~\ref{prop:ordConvTreeSpace} says that any ordinarily convex subset $\frT(\cL)$ is contained in some cone of the Bergman fan of the complete graph on the node set $\cL$.
Moreover, any relatively interior point of the unique minimal cone containing that set dictates the tree topology.
Now we are equipped to determine the tropical median consensus tree, $c(M)$, of the ultrametrics $D^{(1)}(M),\dots,D^{(m)}(M)$, as an ultrametric.
We call it the \emph{tropical supertree} of $\cT=\{T_1,\dots,T_m\}$.

\begin{theorem}\label{thm:supertree}
  The tropical supertree $c(M)$ of $\cT$ has the same tree topology for $M\gg0$.
  Moreover, the length of each edge of $c(M)$ affinely depends on $M$ for $M\gg0$.
\end{theorem}

\begin{proof}
  By Proposition~\ref{prop:affDep}, there exists $M_0>0$ such that the set $\cS:=\smallSetOf{c(M)}{M>M_0}$ is a point or a open ray in the tree space $\frT(\cL)$.
  In particular, $\cS$ is an ordinarily convex subset of $\frT(\cL)$.
  Now Proposition~\ref{prop:ordConvTreeSpace} implies that all trees in $\cS$ share the same tree topology.
\end{proof}

It follows from the construction that the tropical supertree of trees with the same set of taxa is exactly the tropical median consensus tree from \cite[\S5]{ComaneciJoswig:2205.00036}.
This is the case if and only if each edge length is strictly smaller than the parameter $M$.



Theorem~\ref{thm:supertree} includes an effective method to construct supertrees; this is listed as Algorithm~\ref{alg:tsm}.
We want to assess its worst case complexity.
The value for $M$ selected in the algorithm is a generous upper bound of the bound from Theorem~\ref{thm:affDep}; it can be computed in $O(mn^2)$ time.
Extending the tree distances to the combined set of all taxa $\cL$ needs also $O(mn^2)$ operations.
The dominating step is the construction of the tropical median consensus tree at the very end of the procedure; this is \cite[Theorem~17]{ComaneciJoswig:2205.00036}.
The worst case running time is of order $O(n^4m\log^{2}m)$ by \cite[Corollary~15]{ComaneciJoswig:2205.00036}.
The key idea of that algorithm is to reduce computing tropical medians to a transportation problem, which can be solved by the network simplex algorithm \cite[\textsection 11.5]{AMO}.
We summarize our findings as follows.

\begin{algorithm}[th]
\caption{Tropical supertree method}\label{alg:tsm}
\begin{algorithmic}
\Require $T_1,\dots,T_m$ phylogenetic trees of same height
\Ensure phylogenetic tree on $\bigcup_{k=1}^m \cL(T_k)$
\State $\cL \gets \cup_{k=1}^m \cL(T_k)$
\State $n \gets |\cL|$
\State $M \gets 2^{m+\binom{n}{2}-1}\sum_{k=1}^m\|D^{(k)}\|_1 + 1$
\For{$k\gets 1,2,\dots, m$}
  \State compute $D^{(k)}(M)$ from \eqref{eq:distances}
\EndFor
\State \Return \textsc{TropicalMedianConsensusTree}$(D^{(1)}(M),\dots,D^{(m)}(M))$
\end{algorithmic}
\end{algorithm}

\begin{corollary}
  Assuming $m\geq\binom{n}{2}$, the tropical supertree of $m$ trees on a combined set of $n$ taxa can be computed in $O(n^4m\log^2 m)$ time.
\end{corollary}

\begin{figure}[th]\centering
    
  \begin{tabular}{lcl}
    \small
    \begin{tikzpicture}[sloped]
    \node (A) at (0,0) {A};
    \node (B) at (1,0) {B};
    \node (C) at (2,0) {C};
    \node (D) at (3,0) {D};
    \node (E) at (4,0) {E};

    \node (AB) at (0.5,1) {};
    \node (ABC) at (1.25,2) {};
    \node (ABCD) at (2.125,3) {};
    \node (all) at (3.0625,4) {};
    \node (root) at (3.0625,4.25) {};
    
    \draw  (A) |- (AB.center);
    \draw  (B) |- (AB.center);
    \draw  (C) |- (ABC.center);
    \draw  (D) |- (ABCD.center);
    \draw  (E) |- (all.center);
    \draw  (AB.center) |- (ABC.center);
    \draw  (ABC.center) |- (ABCD.center);
    \draw  (ABCD.center) |- (all.center);
    \draw  (all.center) |- (root);
    
    
    \draw (-0.5,0) -- (-0.5,4);
    
    \foreach \y in {0,1,2,3,4}                     
    \draw[shift={(0,\y)},color=black] (-0.5,0) -- (-0.6,0);
    
    \foreach \y in {0,1,2,3,4}   
    \node[left] at (-0.6,\y) {$\y$} ;
    
\end{tikzpicture}
    & \hphantom{xxx} & \begin{tikzpicture}[sloped]
    \node (A) at (0.5,0) {A};
    \node (C) at (1.5,0) {C};
    \node (D) at (3,0) {D};
    \node (E) at (4,0) {E};

    \node (ABC) at (1,2) {};
    \node (ABCD) at (2,3) {};
    \node (all) at (3,4) {};
    \node (root) at (3,4.25) {};
    
    \draw  (A) |- (ABC.center);
    \draw  (C) |- (ABC.center);
    \draw  (D) |- (ABCD.center);
    \draw  (E) |- (all.center);
    \draw  (ABC.center) |- (ABCD.center);
    \draw  (ABCD.center) |- (all.center);
    \draw  (all.center) |- (root);
    
    
    \draw (0,0) -- (0,4);
    
    \foreach \y in {0,1,2,3,4}                     
    \draw[shift={(0,\y)},color=black] (0,0) -- (-0.1,0);
    
    \foreach \y in {0,1,2,3,4}   
    \node[left] at (-0.1,\y) {$\y$} ;
    
\end{tikzpicture} \\
    \multicolumn{1}{c}{(a)}  && \multicolumn{1}{c}{(b)} \\
    \begin{tikzpicture}[sloped]
    \node (A) at (0,0) {A};
    \node (B) at (1,0) {B};
    \node (C) at (2,0) {C};
    \node (D) at (3,0) {D};

    \node (AB) at (0.5,2) {};
    \node (ABC) at (1.25,3) {};
    \node (ABCD) at (2.125,4) {};
    \node (root) at (2.125,4.25) {};
    
    \draw  (A) |- (AB.center);
    \draw  (B) |- (AB.center);
    \draw  (C) |- (ABC.center);
    \draw  (D) |- (ABCD.center);
    \draw  (AB.center) |- (ABC.center);
    \draw  (ABC.center) |- (ABCD.center);
    \draw  (ABCD.center) |- (root);
    
    
    \draw (-0.5,0) -- (-0.5,4);
    
    \foreach \y in {0,1,2,3,4}                     
    \draw[shift={(0,\y)},color=black] (-0.5,0) -- (-0.6,0);
    
    \foreach \y in {0,1,2,3,4}   
    \node[left] at (-0.6,\y) {$\y$} ;
    
\end{tikzpicture}
    && \begin{tikzpicture}[sloped]
    \node (A) at (0,0) {A};
    \node (C) at (1,0) {C};
    \node (D) at (2,0) {D};
    \node (B) at (3,0) {B};
    \node (E) at (4,0) {E};

    \node (AC) at (0.5,3) {};
    \node (ACD) at (1.25,4) {};
    \node (all) at (3,5) {};
    \node (root) at (3,5.25) {};
    
    \draw  (A) |- (AC.center);
    \draw  (B) |- (all.center);
    \draw  (C) |- (AC.center);
    \draw  (D) |- (ACD.center);
    \draw  (E) |- (all.center);
    \draw  (AC.center) |- (ACD.center);
    \draw  (ACD.center) |- (all.center);
    \draw  (all.center) |- (root);
    
    
    \draw (-0.5,0) -- (-0.5,5);
    
    \foreach \y in {0,3,4,5}                     
    \draw[shift={(0,\y)},color=black] (-0.5,0) -- (-0.6,0);
    
    \node[left] at (-0.6,0) {$0$} ;
    \node[left] at (-0.6,1.5) {$\vdots$} ;
    \node[left] at (-0.6,3) {$M-2$} ;
    \node[left] at (-0.6,4) {$M-1$} ;
    \node[left] at (-0.6,5) {$M$} ;
		
\end{tikzpicture} \\
    \multicolumn{1}{c}{(c)}  &&  \multicolumn{1}{c}{(d)}
    \end{tabular}

    \caption{Three trees (a), (b), (c) and their tropical supertree (d) in $\RR^{\tbinom{5}{2}}/\RR1$}
    \label{fig:trees}
\end{figure}

We conclude this section by showing that tropical supertrees have good properties.
Two subsets $X,Y$ of taxa in a phylogenetic tree $T$ form a \emph{nesting}, denoted by $X<Y$, if $\lca_{T}(X)$ is a strict descendant of $\lca_{T}(X\cup Y)$.
In terms of the ultrametric $D$ of $T$, that condition is equivalent to
\begin{equation}\label{eq:halfspace}
  \max_{i,j\in X} D_{ij}<\max_{k,\ell\in X\cup Y}D_{k\ell} \enspace .
\end{equation}
Note that \eqref{eq:halfspace} defines an open tropical half-space \cite[Chapter~6]{ETC}.
Consequently, the set $\frT(\cL;X<Y)$ of trees with the nesting $X<Y$ is tropically convex.
Note that rooted triplets are particular cases of nestings.
A supertree method is called \emph{Pareto on nestings} is the output tree displays any nesting that appears in all the input trees; cf.\ \cite{propSupertree}.

\begin{proposition}
  The tropical supertree method is Pareto on nestings.
\end{proposition}

\begin{proof}
  Consider any nesting $X<Y$ which appears in all the input trees $T_1,\dots,T_m$.
  Then the corresponding ultrametrics $D^{(1)}(M),\dots,D^{(m)}(M)$ on $\cL$, the combined set of taxa, lie in $\frT(\cL;X<Y)$.
  As $\frT(\cL;X<Y)$ is tropically convex, the tropical supertree $c(M)$ lies in $\frT(\cL;X<Y)$ for $M$ sufficiently large.
  This shows that $c(M)$ displays the nesting $X<Y$.
\end{proof}

Figure~\ref{fig:trees} shows a tropical supertree generated from three trees on a total of five taxa.

\section{Conclusion}
\noindent
The proof of Theorem~\ref{thm:affDep}, about the dependence of the tropical Fermat--Weber set $\FW(V(M))$ on the parameter $M$, exploits basic complexity bounds concerning triangulations of products of simplices.
It seems worth-while to have a closer look at the secondary fan:
\begin{question}
  What are the possible combinatorial types of $\FW(V(M))$ as $M$ arbitrarily varies over the real line?
\end{question}
Applied to phylogenetic trees, the answer to the question above will also give a better understanding of the possible tree topologies of the tropical supertree $c(M)$ as $M$ varies.
If $M$ is seen as a measure of uncertainty, the knowledge of tree topologies before stabilization might strengthen or reveal new relationships between the taxa.

In view of \cite{ArdilaKlivans:2006} the restriction to equidistant trees is natural in the tropical setting.
Yet from an application point of view this is quite restrictive.
\begin{question}
  Does the tropical supertree method admit a natural generalization to non-equidistant trees?
\end{question}
This would mean to pass to general tree metrics, which would blur the role of the root.
Hence, such a generalization should probably consider unrooted metric trees.
In that case, we can optimize over the tropical Grassmannian $\TGr(2,n)$, where $n$ is the total number of taxa; cf.\ \cite[\S4.3]{Tropical+Book} and \cite[\S10.6]{ETC}.
Unlike $\frT(\cL)$, the tropical Grassmannian is not tropically convex, resulting in a constrained Fermat--Weber problem.
This may lead to a very high complexity and might be intrinsic.
This should be compared to the compatibility problem of phylogenetic trees, which is $\textsf{NP}$-complete for the unrooted case, while the rooted cases can be solved in polynomial time \cite[Chapter~6]{SempleSteel:2003}.
Another approach could employ tight spans of finite metric spaces in the sense of Dress \cite{Dress:1984}; cf.\ \cite[\S10.7]{ETC} and \cite[\S7.4]{SempleSteel:2003}.


\printbibliography

@book{ETC,
  author = {Joswig, Michael},
  title = {Essentials of tropical combinatorics},
  publisher = {American Mathematical Society},
  address = {Providence, RI},
  series = {Graduate Studies in Mathematics},
  volume = {219},
  year = {2021},
}

@article {ArdilaKlivans:2006,
    AUTHOR = {Ardila, Federico and Klivans, Caroline J.},
     TITLE = {The {B}ergman complex of a matroid and phylogenetic trees},
   JOURNAL = {J. Combin. Theory Ser. B},
  FJOURNAL = {Journal of Combinatorial Theory. Series B},
    VOLUME = {96},
      YEAR = {2006},
    NUMBER = {1},
     PAGES = {38--49},
      ISSN = {0095-8956},
   MRCLASS = {05B35},
  MRNUMBER = {2185977},
MRREVIEWER = {Neil L. White},
       DOI = {10.1016/j.jctb.2005.06.004},
       URL = {https://doi.org/10.1016/j.jctb.2005.06.004},
}

@book {BMS:1999,
    AUTHOR = {Boltyanski, Vladimir and Martini, Horst and Soltan, Valeriu},
     TITLE = {Geometric methods and optimization problems},
    SERIES = {Combinatorial Optimization},
    VOLUME = {4},
 PUBLISHER = {Kluwer Academic Publishers, Dordrecht},
      YEAR = {1999},
      ISBN = {0-7923-5454-0},
   MRCLASS = {90-02 (49K15 65D18 90B85)},
  MRNUMBER = {1677397},
MRREVIEWER = {Anita Kripfganz},
       DOI = {10.1007/978-1-4615-5319-9},
       URL = {https://doi.org/10.1007/978-1-4615-5319-9},
}

@book {DLRS,
    AUTHOR = {De Loera, Jes\'{u}s A. and Rambau, J\"{o}rg and Santos, Francisco},
     TITLE = {Triangulations},
    SERIES = {Algorithms and Computation in Mathematics},
    VOLUME = {25},
      NOTE = {Structures for algorithms and applications},
 PUBLISHER = {Springer-Verlag, Berlin},
      YEAR = {2010},
      ISBN = {978-3-642-12970-4},
   MRCLASS = {52B55 (05C10 52B05 57Q15 68U05)},
  MRNUMBER = {2743368},
       DOI = {10.1007/978-3-642-12971-1},
       URL = {https://doi.org/10.1007/978-3-642-12971-1},
}

@book {Schrijver:1986,
    AUTHOR = {Schrijver, Alexander},
     TITLE = {Theory of linear and integer programming},
    SERIES = {Wiley-Interscience Series in Discrete Mathematics},
      NOTE = {A Wiley-Interscience Publication},
 PUBLISHER = {John Wiley \& Sons, Ltd., Chichester},
      YEAR = {1986},
      ISBN = {0-471-90854-1},
   MRCLASS = {90C05 (90C10)},
  MRNUMBER = {874114},
MRREVIEWER = {J\"{u}rgen K\"{o}hler},
}

@article{Page+Yoshida+Zhang:2020,
    author = {Page, Robert and Yoshida, Ruriko and Zhang, Leon},
    title = {Tropical principal component analysis on the space of phylogenetic trees},
    journal = {Bioinformatics},
    volume = {36},
    number = {17},
    pages = {4590-4598},
    year = {2020},
    month = {06},
    issn = {1367-4803},
    doi = {10.1093/bioinformatics/btaa564},
    url = {https://doi.org/10.1093/bioinformatics/btaa564},
    __eprint = {https://academic.oup.com/bioinformatics/article-pdf/36/17/4590/34220689/btaa564.pdf},
}

@book {SempleSteel:2003,
    AUTHOR = {Semple, Charles and Steel, Mike},
     TITLE = {Phylogenetics},
    SERIES = {Oxford Lecture Series in Mathematics and its Applications},
    VOLUME = {24},
 PUBLISHER = {Oxford University Press, Oxford},
      YEAR = {2003},
      ISBN = {0-19-850942-1},
   MRCLASS = {92D15 (05C05 05C90 92D40)},
  MRNUMBER = {2060009},
MRREVIEWER = {Vincent L. Moulton},
}

@incollection {Bryant:2003,
    AUTHOR = {Bryant, David},
     TITLE = {A classification of consensus methods for phylogenetics},
 BOOKTITLE = {Bioconsensus ({P}iscataway, {NJ}, 2000/2001)},
    SERIES = {DIMACS Ser. Discrete Math. Theoret. Comput. Sci.},
    VOLUME = {61},
     PAGES = {163--183},
 PUBLISHER = {Amer. Math. Soc., Providence, RI},
      YEAR = {2003},
   MRCLASS = {92B10 (05C05 91B12 92D15)},
  MRNUMBER = {1982426},
}

@book {Tropical+Book,
    AUTHOR = {Maclagan, Diane and Sturmfels, Bernd},
     TITLE = {Introduction to tropical geometry},
    SERIES = {Graduate Studies in Mathematics},
    VOLUME = {161},
 PUBLISHER = {American Mathematical Society, Providence, RI},
      YEAR = {2015},
      ISBN = {978-0-8218-5198-2},
   MRCLASS = {14T05 (05B35 15A80 52B70)},
  MRNUMBER = {3287221},
}

@ARTICLE{Develin+Sturmfels:2004,
  author = {Develin, Mike and Sturmfels, Bernd},
  title = {Tropical convexity},
  journal = {Doc. Math.},
  year = {2004},
  volume = {9},
  pages = {1--27 (electronic)},
  note = {correction: ibid., pp.\ 205--206},
  fjournal = {Documenta Mathematica},
  issn = {1431-0635},
  mrclass = {52A30 (52B10 92D15)},
  mrnumber = {MR2054977 (2005i:52010)},
  mrreviewer = {Gerard Sierksma}
}

@article{Bryant+Francis+Steel:2017,
    author = {Bryant, David and Francis, Andrew and Steel, Mike},
    title = "{Can We \enquote{Future-Proof} Consensus Trees?}",
    journal = {Systematic Biology},
    volume = {66},
    number = {4},
    pages = {611-619},
    year = {2017},
    month = {02},
    abstract = "{Consensus methods are widely used for combining phylogenetic trees into a single estimate of the evolutionary tree for a group of species. As more taxa are added, the new source trees may begin to tell a different evolutionary story when restricted to the original set of taxa. However, if the new trees, restricted to the original set of taxa, were to agree exactly with the earlier trees, then we might hope that their consensus would either agree with or resolve the original consensus tree. In this article, we ask under what conditions consensus methods exist that are “future proof” in this sense. While we show that some methods (e.g., Adams consensus) have this property for specific types of input, we also establish a rather surprising “no-go” theorem: there is no “reasonable” consensus method that satisfies the future-proofing property in general. We then investigate a second notion of “future proofing” for consensus methods, in which trees (rather than taxa) are added, and establish some positive and negative results. We end with some questions for future work.}",
    issn = {1063-5157},
    doi = {10.1093/sysbio/syx030},
    url = {https://doi.org/10.1093/sysbio/syx030},
    _eprint = {https://academic.oup.com/sysbio/article-pdf/66/4/611/17726420/syx030.pdf},
}

@Misc{ComaneciJoswig:2205.00036,
  author = {Com\u{a}neci, Andrei and Joswig, Michael},
  title =  {Tropical medians by transportation},
  year =   {2022},
  note =   {Preprint \arXiv{2205.00036}}
}

@article{propSupertree,
    author = {Wilkinson, Mark and Cotton, James A. and Lapointe, François-Joseph and Pisani, Davide},
    title = {Properties of supertree methods in the consensus setting},
    journal = {Systematic Biology},
    volume = {56},
    number = {2},
    pages = {330-337},
    year = {2007},
    month = {04},
    issn = {1063-5157},
    doi = {10.1080/10635150701245370},
    __url = {https://doi.org/10.1080/10635150701245370},
    __eprint = {https://academic.oup.com/sysbio/article-pdf/56/2/330/24203813/56-2-330.pdf},
}

@book {AMO,
    AUTHOR = {Ahuja, Ravindra K. and Magnanti, Thomas L. and Orlin, James B.},
     TITLE = {Network flows},
      NOTE = {Theory, algorithms, and applications},
 PUBLISHER = {Prentice Hall, Inc.},
   ADDRESS = {Englewood Cliffs, NJ},
      YEAR = {1993},
      ISBN = {0-13-617549-X},
   MRCLASS = {90B10 (90-01 90-02 90C35)},
  MRNUMBER = {1205775},
MRREVIEWER = {Jacques A. Ferland},
}

@article {Brandeau+Chiu:91,
    AUTHOR = {Brandeau, Margaret L. and Chiu, Samuel S.},
     TITLE = {Parametric analysis of optimal facility locations},
   JOURNAL = {Networks},
  FJOURNAL = {Networks. An International Journal},
    VOLUME = {21},
      YEAR = {1991},
    NUMBER = {2},
     PAGES = {223--243},
      ISSN = {0028-3045},
   MRCLASS = {90B80},
  MRNUMBER = {1093841},
       DOI = {10.1002/net.3230210207},
       URL = {https://doi.org/10.1002/net.3230210207},
}

@article {Dress:1984,
    AUTHOR = {Dress, Andreas W. M.},
     TITLE = {Trees, tight extensions of metric spaces, and the cohomological dimension of certain groups: a note on combinatorial properties of metric spaces},
   JOURNAL = {Adv. in Math.},
  FJOURNAL = {Advances in Mathematics},
    VOLUME = {53},
      YEAR = {1984},
    NUMBER = {3},
     PAGES = {321--402},
      ISSN = {0001-8708},
   MRCLASS = {05C10 (20F32 20J05 54C25 54F50)},
  MRNUMBER = {753872},
MRREVIEWER = {Peter J. Slater},
       DOI = {10.1016/0001-8708(84)90029-X},
       URL = {https://doi.org/10.1016/0001-8708(84)90029-X},
}

@article{Kapli+Yang+Telford:2020,
  author = 	 {Kapli, Paschalia and Yang, Ziheng and Telford, Maximilian J.},
  title = 	 {Phylogenetic tree building in the genomic age},
  journal = 	 {Nature Reviews Genetics},
  year = 	 2020,
  volume = 	 21,
  pages = 	 {428-444},
  doi = {10.1038/s41576-020-0233-0},
}

@article{li20:_phylog_sars_cov,
  author = 	 {Li, T. and Liu, D. and Yang, Y. and Guo, J. and Feng, Y. and Zhang, X. and Cheng, S. and Feng, Y.},
  title = 	 {Phylogenetic supertree reveals detailed evolution of {SARS-CoV-2}},
  journal = 	 {Sci. Rep.},
  year = 	 2020,
  volume = 	 10,
  number = 	 22366,
  doi = {10.1038/s41598-020-79484-8},
}

\end{document}